\documentclass[11pt]{article}
\usepackage{amsmath, amscd, amssymb, latexsym, epsfig, color, amsthm,enumerate}
\usepackage[all]{xy}
\usepackage{verbatim}
\numberwithin{equation}{section}

\newtheorem{theorem}{Theorem}[section]
\newtheorem{proposition}[theorem]{Proposition}

\newtheorem{corollary}[theorem]{Corollary}

\newtheorem{lemma}[theorem]{Lemma}

\newtheorem{problem}[theorem]{Problem}

\theoremstyle{definition}
\newtheorem{definition}[theorem]{Definition}

\newtheorem{remark}[theorem]{Remark}

\DeclareMathOperator{\skel}{Skel}
\DeclareMathOperator\lk{\mathrm{lk}}

\newcommand{\R}{{\mathbb R}}

\newcommand{\C}{{\mathcal C}}

\title{Highly neighborly centrally symmetric spheres}
\author{
	Isabella Novik\thanks{Research is partially\textsl{}
		supported by NSF grant DMS-1664865, and by Robert R.~\&  Elaine F.~Phelps Professorship in Mathematics. }\\
	\small Department of Mathematics\\[-0.8ex]
	\small University of Washington\\[-0.8ex]
	\small Seattle, WA 98195-4350, USA\\[-0.8ex]
	\small \texttt{novik@uw.edu}
	\and 
	Hailun Zheng \\
	\small Department of Mathematics\\[-0.8ex]
	\small University of Michigan\\[-0.8ex]
	\small Ann Arbor, Michigan 48109-1043, USA \\[-0.8ex]
	\small \texttt{hailunz@umich.edu}
}
\begin{document}
	\maketitle
	
	\begin{abstract} 
		In 1995, Jockusch constructed an infinite family of centrally symmetric $3$-dimensional simplicial spheres that are cs-$2$-neighborly. Here we generalize his construction and show that for all $d\geq 3$ and $n\geq d+1$, there exists a centrally symmetric $d$-dimensional simplicial sphere with $2n$ vertices that is cs-$\lceil d/2\rceil$-neighborly. This result combined with work of Adin and Stanley completely resolves the upper bound problem for centrally symmetric simplicial spheres.
	\end{abstract}
	
	\section{Introduction}
	In this paper we construct highly neighborly centrally symmetric $d$-dimensional spheres with an arbitrarily large even number of vertices. A simplicial complex is {\em centrally symmetric} ({\em cs}, for short) if it possesses a free simplicial involution. We refer to a pair of vertices that form an orbit under this involution as antipodes or antipodal vertices. One large class of examples is given by the boundary complexes of cs simplicial polytopes: a polytope $P$ is cs if $P=-P$; the involution, in this case, is induced by the map $v\mapsto-v$ on the vertices.
	
	A (non-cs) simplicial complex is called {\em $\ell$-neighborly} if every $\ell$ of its vertices form a face. For instance, the boundary complex  of the $(d+1)$-dimensional simplex is $(d+1)$-neighborly, while the boundary complex of the $(d+1)$-dimensional cyclic polytope with $n\geq d+3$ vertices, $C(d+1,n)$, is $\lceil d/2\rceil$-neighborly. The interest in neighborly polytopes arises from the celebrated upper bound theorem \cite{McMullen70, Stanley75} asserting that among all  $d$-dimensional simplicial spheres  with $n$ vertices, the boundary complex of $C(d+1,n)$ simultaneously maximizes all the face numbers. The cyclic polytope in this statement can be replaced with any $\lceil d/2\rceil$-neighborly $d$-dimensional simplicial sphere --- the objects that abound in nature, see  \cite{Padrol-13}.
	
	This notion of $\ell$-neighborliness can be easily modified for the class of cs complexes: a cs simplicial complex $\Delta$ is {\em cs-$\ell$-neighborly} if every set of $\ell$ of its vertices, no two of which are antipodes, is a face of $\Delta$. Furthermore, the same notion applies to any (not necessarily cs) full-vertex subcomplex $\Gamma$ of $\Delta$. For instance, the boundary complex of the $(d+1)$-dimensional cross-polytope is cs-$(d+1)$-neighborly, while the boundary complex of the same cross-polytope with one facet removed is cs-$d$-neighborly.  Adin \cite{Adin91} and Stanley (unpublished) proved that in a complete analogy with Stanley's upper bound theorem, among all cs simplicial spheres of dimension $d$ and with $2n$ vertices, a cs-$\lceil d/2\rceil$-neighborly sphere simultaneously maximizes all the face numbers, {\em assuming such a sphere exists}. (See \cite{N03,N05} for an extension of this result to cs simplicial manifolds.)

	Thus, two natural questions to consider are:  do there exist cs simplicial polytopes of dimension $d+1\geq 4$ with arbitrarily many vertices whose boundary complexes are cs-$\lceil d/2\rceil$-neighborly? Do there exist cs simplicial spheres of dimension $d\geq 3$ with arbitrarily many  vertices that are cs-$\lceil d/2\rceil$-neighborly? 
	
	The answer to the first question was given by McMullen and Shephard \cite{McMShep} more than fifty years ago: extending the $4$-dimensional case worked out by Gr\"unbaum \cite[p.~116]{Gru-book}, they proved that while there do exist cs $(d+1)$-dimensional polytopes with $2(d+2)$ vertices that are cs-$\lceil d/2\rceil$-neighborly, a cs  $(d+1)$-dimensional polytope with $2(d+3)$ vertices cannot be more than cs-$\lfloor (d+2)/3\rfloor$-neighborly. Moreover, according to \cite{LinNov}, a cs $(d+1)$-dimensional polytope with more than $2^{d+1}$ vertices cannot be even cs-$2$-neighborly. 
	
	The second question remained a total mystery until in 1995 Jockusch \cite{Jockusch95} showed that, in a sharp contrast with the situation for cs $4$-dimensional polytopes, for {\em every} value of $n\geq 4$, there   exists a cs $3$-dimensional sphere with $2n$ vertices that is cs-$2$-neighborly. In addition, for $d\leq 6$, Lutz \cite{Lutz} found (by a computer search)  several cs $d$-dimensional spheres with $2(d+3)$ vertices that are  cs-$\lceil d/2\rceil$-neighborly.
	
	Here, we build on work of Jockusch to provide a complete answer to the second question:  for {\em all} values of $d\geq 3$ and $n\geq d+1$, there exists a cs $d$-dimensional combinatorial sphere with $2n$ vertices, $\Delta^{d}_{n}$, that is cs-$\lceil d/2\rceil$-neighborly. Thus, our result combined with work of Adin and Stanley completely resolves the upper bound problem for cs simplicial spheres. (At the same time, there is not even a plausible upper bound conjecture for cs polytopes.) 
	
	Our construction is by induction on both $d$ and $n$. The key idea in constructing $\Delta^{d}_{n+1}$ from $\Delta^{d}_n$ is to define for each $i\leq \lceil d/2 \rceil -1$, an auxiliary $d$-dimensional ball with $2n$ vertices, $B^{d,i}_n\subset\Delta^{d}_{n}$, that is both $i$-stacked and cs-$i$-neighborly, see Sections 2 and 3 for definitions. For $d=3$, our construction reduces to Jockusch's construction. It is worth mentioning that using the same balls in fact allows us to construct for any $\ell\leq \lceil d/2\rceil$, a family of cs $d$-dimensional combinatorial spheres that are cs-$\ell$-neighborly but not cs-$(\ell+1)$-neighborly.
	
	The structure of the paper is as follows. In Section 2 we discuss several basic facts and definitions pertaining to simplicial complexes and PL topology. 
Sections 3 is a hard duty section that contains our inductive construction and the proof that this construction works. In Section 4, we study some other properties of the spheres $\Delta^d_n$.  We close in Section 5 with several remarks and open questions.

\section{Preliminaries}
In this section we review some background related to simplicial complexes and PL topology.  For all undefined terminology we refer the readers to \cite{Bjorner95}.

A \emph{simplicial complex} $\Delta$ with vertex set $V=V(\Delta)$ is a collection of subsets of $V$ that is closed under inclusion and contains all singletons: $\{v\}\in\Delta$ for all $v\in V$. The elements of $\Delta$ are called {\em faces}. The \emph{dimension of a face} $\tau\in\Delta$ is $\dim\tau:=|\tau|-1$. The \emph{dimension of $\Delta$}, $\dim\Delta$, is the maximum dimension of its faces. A face of a simplicial complex $\Delta$ is a \textit{facet} if it is maximal w.r.t.~inclusion. We say that $\Delta$ is \emph{pure} if all facets of $\Delta$ have the same dimension.

Let $V$ be a set of size $d+1$. Two fundamental examples of pure simplicial complexes are the $d$-dimensional simplex on $V$, $\overline{V}:=\{\tau \ : \ \tau\subseteq V\}$, and its boundary complex, $\partial\overline{V}:=\{\tau : \tau\subsetneq V\}$.


To simplify the notation, for a face that is a vertex, an edge, or a triangle, we write $v$, $uv$, and $uvw$ instead of $\{v\}$, $\{u,v\}$, and $\{u,v,w\}$, respectively. We denote by $(v_1,v_2,\dots, v_n)$ a path with edges $v_1v_2, v_2v_3,\dots, v_{n-1}v_n$ if $v_n\neq v_1$, or a cycle if $v_n=v_1$. In particular, a path of length one $(v_1, v_2)$ is a $1$-dimensional simplex, so it can also be written as $\overline{v_1v_2}$. 

Let $\Delta$ be a simplicial complex. The \emph{$k$-skeleton} of $\Delta$, $\skel_k(\Delta)$, is the subcomplex of $\Delta$ consisting of all faces of dimension $\leq k$. If $\tau$ is a face of $\Delta$, then 
the  {\em link of $\tau$ in $\Delta$} is the following subcomplex of $\Delta$:
\[
\lk(\tau,\Delta):= \{\sigma\in \Delta \ : \ \sigma\cap\tau=\emptyset \mbox{ and } \sigma\cup\tau\in\Delta\}.
\]
Finally, if $\Delta$ is pure and $\Gamma$ is a full-dimensional pure subcomplex of $\Delta$, then $\Delta\backslash \Gamma$ is the subcomplex of $\Delta$ generated by those facets of $\Delta$ that are not in $\Gamma$. 

If $\Delta$ and $\Gamma$ are simplicial complexes on disjoint vertex sets, then the \textit{join} of $\Delta$ and $\Gamma$ is the simplicial complex $\Delta*\Gamma = \{\sigma \cup \tau \ : \ \sigma \in \Delta \text{ and } \tau \in \Gamma\}$. Two important special cases are the \emph{cone} over $\Delta$ with apex $v$ defined as  the join $\Delta\ast\overline{v}$ and the \emph{suspension} of~$\Delta$, $\Sigma\Delta$, defined as the join of $\Delta$ with a $0$-dimensional sphere. In the rest of the paper, we write $\Delta \ast v$ in place of  $\Delta\ast\overline{v}$. 

Let $V$ be a set of size $d+1$ and let $\overline{V}$ be the $d$-dimensional simplex on $V$. A \emph{combinatorial $d$-ball} is a simplicial complex PL homeomorphic to $\overline{V}$. Similarly, a \emph{combinatorial $(d-1)$-sphere} is a simplicial complex PL homeomorphic to $\partial\overline{V}$.

One advantage of working with combinatorial balls and spheres is that they satisfy several natural properties that fail in the class of simplicial balls and spheres. For instance, if $\Delta$ is a combinatorial $d$-sphere and $\Gamma\subset\Delta$ is a combinatorial $d$-ball, then so is $\Delta\backslash \Gamma$, see \cite{Hudson}. Furthermore, the link of any face in a combinatorial sphere is a combinatorial sphere. On the other hand, the link of a face $\tau$ in a combinatorial $d$-ball $B$ is either a combinatorial ball or a combinatorial sphere; in the former case we say that $\tau$ is  a \emph{boundary face} of $B$, and in the latter case that $\tau$ is an \emph{interior face} of $B$. The \emph{boundary complex} of $B$, $\partial B$, is the subcomplex of $B$ that consists of all boundary faces of $B$; in particular, $\partial B$ is a combinatorial $(d-1)$-sphere.

The following lemma summarizes a few basic but very useful properties of combinatorial balls. We only prove the last part; the proofs of the  first two parts along with additional information on PL topology can be found in \cite{Hudson}, see also \cite{Lickorish}.
\begin{lemma}\label{lm: interior faces of join} \qquad
	\begin{enumerate} 
		\item Let $B$ and $B'$ be combinatorial balls. Then $B*B'$ is a combinatorial ball; its interior faces are sets of the form $F\cup F'$, where $F$ is an interior face of $B$ and $F'$ is an interior face of $B'$. Furthermore, the cone over $\partial B$, $\partial B * v$, is a combinatorial ball; its boundary complex is $\partial B$.
		\item Let $B$ and $B'$ be combinatorial $d$-balls such that $B\cap B'\subseteq \partial B \cap \partial B'$ is a combinatorial $(d-1)$-ball. Then $B\cup B'$ is a combinatorial $d$-ball. The set of interior faces of $B\cup B'$ consists of the interior faces of $B$, the interior faces of $B'$, and the interior faces of $B\cap B'$.
		\item Assume that combinatorial balls $B$ and $B'$ are full-dimensional subcomplexes of a combinatorial sphere $\Gamma$. If $B$ and $B'$ share a common interior face, then they share a facet.
	\end{enumerate}
\end{lemma}
\begin{proof}
	For part 3, let $\tau$ be a common interior face of $B$ and $B'$. Assume that $\dim\Gamma=d$ and $\dim\tau=k$. Then the link of $\tau$ in $B$ is a combinatorial $(d-k-1)$-sphere, and so is the link of $\tau$ in $\Gamma$. Furthermore, the link of $\tau$ in $B$ is contained in the link of $\tau$ in $\Gamma$. Thus, these two links must be equal. In particular, every facet $F$ of $\Gamma$ containing $\tau$ must be a facet of $B$. By the same argument, such an $F$ must also be a facet of $B'$. The result follows.
\end{proof}

A simplicial complex $\Delta$ is \textit{centrally symmetric} or {\em cs} if its vertex set is endowed with a {\em free involution} $\alpha: V(\Delta) \rightarrow V(\Delta)$ that induces a free involution on the set of all nonempty faces of $\Delta$. In more detail,  for all nonempty faces $\tau \in \Delta$, the following holds: $\alpha(\tau)\in\Delta$, $\alpha(\tau)\neq \tau$, and $\alpha(\alpha(\tau))=\tau$. To simplify notation, we write $\alpha(\tau)=-\tau$ and refer to $\tau$ and $-\tau$ as {\em antipodal faces} of $\Delta$. Similarly, if $\Gamma$ is a subcomplex of $\Delta$ we write $-\Gamma$ in place of $\alpha(\Gamma)$.

One example of a cs combinatorial $d$-sphere is the boundary complex of the $(d+1)$-dimensional \emph{cross-polytope}, $\partial\C^*_{d+1}$. The polytope $\C_{d+1}^*$ is the convex hull of $\{\pm e_1,\pm e_2,\ldots, \pm e_{d+1}\}$, where $e_1,e_2,\ldots, e_{d+1}$ are the endpoints of the standard basis in $\R^{d+1}$. As an abstract simplicial complex, $\partial\C^*_{d+1}$ is the $(d+1)$-fold suspension of $\{\emptyset\}$. Equivalently, it is the collection of all subsets of $V_{d+1}:=\{\pm v_1,\ldots,\pm v_{d+1}\}$ that contain at most one vertex from each pair $\{\pm v_i\}$. In particular, every cs simplicial complex with vertex set $V_n$ is a subcomplex of $\partial\C^*_n$.  

We close this section with a discussion of neighborliness and stackedness. Let $\Delta\subseteq \partial\C^*_{n}$ be a simplicial complex, possibly non-cs, and let $1\leq i\leq n$. We say that $\Delta$ is \emph{cs-$i$-neighborly} (w.r.t.~$V_n$), if $\skel_{i-1}(\Delta)=\skel_{i-1}(\partial \C_n^*)$. For $i=1$, this simply means that $V(\Delta)=V_n$. For convenience, we also refer to  simplices (i.e., faces of $\partial\C^*_{n}$) as cs-$0$-neighborly complexes.


A combinatorial $d$-ball $B$ is called \emph{$i$-stacked} (for some $0\leq i\leq d$), if all interior faces of $B$ are of dimension $\geq d-i$, that is, $\skel_{d-i-1}(B)=\skel_{d-i-1}(\partial B)$. In particular, $0$-stacked balls are simplices, and $1$-stacked balls are also known as stacked balls. The notion of stackedness takes its origins in the generalized lower bound theorem \cite{McMullenWalkup71,MuraiNevo2013,Stanley80}. 

The following lemmas will be handy.
\begin{lemma}\label{lm: stackedness via union and join}
	Let $B_1$ and $B_2$ be combinatorial balls of dimension $d_1$ and $d_2$, respectively. 
	If $B_1$ is $i_1$-stacked and $B_2$ is $i_2$-stacked, then
	\begin{enumerate}
		\item The complex $B_1*B_2$ is an $(i_1+i_2)$-stacked combinatorial $(d_1+d_2+1)$-ball. 
		\item Furthermore, if $d_1=d_2=d$, $i_1\leq i_2$, and  $B_1\cap B_2\subseteq \partial B_1\cap \partial B_2$ is a combinatorial $(d-1)$-ball that is $i_3$-stacked for some $i_3< i_2$, then $B_1\cup B_2$ is an $i_2$-stacked combinatorial $d$-ball.
	\end{enumerate}
\end{lemma}
\begin{proof}
	For part 1, observe that by definition of stackedness, all interior faces of $B_j$ have dimension  $\geq d_j-i_j$ for $j=1,2$. Hence by Lemma \ref{lm: interior faces of join}, the interior faces of $B_1*B_2$, which is a combinatorial $(d_1+d_2+1)$-ball, have dimension $\geq (d_1+d_2)-(i_1+i_2)+1$. Thus, $B_1*B_2$ is $(i_1+i_2)$-stacked. 
	
	Part 2 similarly follows from Lemma \ref{lm: interior faces of join} and the definition of stackedness. Indeed, the interior faces of both $B_1$ and $B_2$ have dimension $\geq d-i_2$. All other interior faces of $B_1\cup B_2$ are interior faces of $B_1\cap B_2$, and so have dimension $\geq (d-1)-i_3\geq d-i_2$. Hence $B_1\cup B_2$ is $i_2$-stacked.
\end{proof}

\begin{lemma}\label{lm: property of complement}
Let $k\geq 1$ be an integer. Let $\Delta\subseteq \partial \C^*_n$ be a combinatorial $(2k-1)$-sphere that is cs-$k$-neighborly w.r.t.~$V_n$, and let $B\subseteq \Delta $ be a combinatorial $(2k-1)$-ball that is both cs-$(k-1)$-neighborly w.r.t.~$V_n$ and $(k-1)$-stacked. Then $\Delta \backslash B$ is a combinatorial $(2k-1)$-ball that is cs-$k$-neighborly and $k$-stacked.
\end{lemma}
\begin{proof}
Let $F\in\partial\C^*_n$ be any set of size $\leq k$.
	Since $\Delta$ is cs-$k$-neighborly, $F$ is a face of $\Delta$. Thus, if $F\notin B$, then $F\in \Delta\backslash B$. On the other hand, if $F\in B$, then since $\dim F\leq k-1<(2k-1)-(k-1)$ and since $B$ is $(k-1)$-stacked, $F$ must be a boundary face of $B$ and thus also a face of $\Delta\backslash B$. We infer that $\Delta\backslash B$ is cs-$k$-neighborly. 
	
	Next let $F$ be a face of $\Delta\backslash B$ of dimension $<(2k-1)-k=k-1$, i.e., $|F|\leq k-1$. Since $B$ is cs-$(k-1)$-neighborly, it follows that $F\in B$. Thus $F$ must lie on the boundary of $B$ and hence also on the boundary of $\Delta \backslash B$. We conclude that $\Delta\backslash B$ is $k$-stacked.
\end{proof}
\section{The construction}
Our goal is to construct $\Delta^{d}_n$ --- a cs combinatorial $d$-sphere with $2n\geq 2d+2$ vertices that is cs-$\lceil d/2\rceil$-neighborly. Our method is to build a certain ball on those $2n$ vertices that is both cs-$(\lceil d/2\rceil-1)$-neighborly and  $(\lceil d/2\rceil-1)$-stacked. The following lemma explains the significance of these balls and outlines the inductive procedure on $n$ we will use once such balls are constructed. For all values of $d+1\leq n$, the vertex set of $\Delta^{d}_n$ will be $V_n=V(\partial\C^*_n)$. In particular, our initial complex $\Delta^{d}_{d+1}$ is $\partial\C^*_{d+1}$.
\begin{lemma}\label{lm: induction method}
	Let $d\geq 1$ and $1\leq i\leq \lceil d/2\rceil$ be integers. Assume that $\Delta^{d,i}_n$ is a cs combinatorial $d$-sphere with $V(\Delta^{d,i}_n)=V_n$ that is cs-$i$-neighborly. Assume further that $B_n^{d,i-1}\subseteq\Delta_n^{d,i}$ is a combinatorial $d$-ball that satisfies the following properties: 
	\begin{itemize}
		\item the ball $B_{n}^{d,i-1}$ is both cs-$(i-1)$-neighborly w.r.t.~$V_n$ and $(i-1)$-stacked, and
		\item the balls $B^{d,i-1}_n$ and $-B^{d,i-1}_n$ share no common facets.
	\end{itemize} 
	Then the complex $\Delta_{n+1}^{d,i}$ obtained from $\Delta^{d,i}_n$ by replacing $B_n^{d,i-1}$ with $\partial B_n^{d,i-1}*v_{n+1}$ and $-B_n^{d,i-1}$ with $\partial(-B_n^{d,i-1})*(-v_{n+1})$ is a cs combinatorial $d$-sphere with $V(\Delta_{n+1}^{d,i})=V_{n+1}$ that is cs-$i$-neighborly.
\end{lemma}
\begin{proof}
	Observe that $B^{d, i-1}_n$ and $\partial B^{d,i-1}_n* v_{n+1}$ are two combinatorial balls with the same boundary. The fact that $B^{d,i-1}_n$ and $-B^{d,i-1}_n$ share no common facets combined with Lemma~\ref{lm: interior faces of join} then implies that $\Delta^{d,i}_{n+1}$ is a combinatorial sphere. Moreover, the definition of $\Delta^{d,i}_{n+1}$ along with the fact that $\Delta^{d,i}_{n}$ is a cs complex guarantees that the complex $\Delta^{d,i}_{n+1}$ is also cs. 
	
	To show that $\Delta^{d,i}_{n+1}$ is cs-$i$-neighborly, consider a set $F\in\partial \C^*_{n+1}$ with $|F|\leq i\leq \lceil d/2\rceil$. If $v_{n+1}$ is in $F$, then  $F\backslash v_{n+1}$ is a face of $\partial\C^*_{n}$ of size at most $i-1$. Since $B^{d,i-1}_n$ is cs-$(i-1)$-neighborly and $(i-1)$-stacked, it follows that
	$$F\backslash v_{n+1}\in \skel_{i-2}(B^{d,i-1}_n)
	\subseteq\skel_{d-i} (B^{d,i-1}_n)\subseteq\partial B^{d,i-1}_n. $$ Hence $F\in \partial B^{d,i-1}_n*v_{n+1}\subseteq \Delta^{d,i}_{n+1}$. If $-v_{n+1}\in F$, then by the above argument, $-F\in\Delta^{d,i}_{n+1}$, and so by symmetry, $F\in\Delta^{d,i}_{n+1}$. Finally, if $\pm v_{n+1} \notin F$, then since $\Delta^{d,i}_n$ is cs-$i$-neighborly, $F\in \Delta^{d,i}_n$. As any face of $\pm B^{d,i-1}_n$ of dimension $\leq i-1$ is on the boundary of $\pm B^{d,i-1}_n$, we conclude that  $F\in \Delta^{d,i}_n\backslash \pm B^{d,i-1}_n\subseteq \Delta^{d,i}_{n+1}$. 
\end{proof}

The idea is to build $B^{d,\lceil d/2\rceil-1}_n$ from balls that are less cs-neighborly using intertwined induction.
\begin{definition} \label{main-constr}
	Let $d\geq 1$, $i\leq \lceil d/2\rceil$, and $n\geq d+1$ be integers. Define $\Delta^{d}_n$ and $B^{d,i}_n$ inductively as follows: 
	\begin{itemize}
	\item For the initial cases, define
	$\Delta^{1}_{n}:=(v_1, v_2,\dots, v_n,-v_1,-v_2,\dots, -v_n,v_1)$, $\Delta^{d}_{d+1}:=\partial \C^*_{d+1}$, 
	$B^{d,j}_n:=\emptyset$   if $j<0$, and $B^{1, 0}_n:=\overline{(-v_1)v_n}$. (In particular, $B^{1,j}_n\subseteq \Delta^{1,1}_n$ for all $j\leq 0$.)

\item
Assume that $\Delta^{d-1}_{m}$ and $B^{d-1,i}_m \subseteq \Delta^{d-1}_{m}$ are already defined for all $i\leq \lfloor (d-1)/2\rfloor$ and $m\geq d$. If $d=2k$, then define $B^{2k-1,k}_{m}:=\Delta^{2k-1}_{m}\backslash B^{2k-1,k-1}_{m}$ (for $m\geq 2k$). Then, for all $n\geq d+1$ and $i\leq \lfloor d/2\rfloor$, define 
	 $$B^{d, i}_n:=
		\left(B^{d-1, i}_{n-1}*v_n\right) \cup\left( (-B^{d-1,i-1}_{n-1})*(-v_{n})\right).$$ 
		
	\item	 
	If $\Delta^{d}_n$ is already defined (and assuming that $\Delta^{d}_n\supseteq B^{d,\lceil d/2\rceil -1}_n$ and also that $B^{d,\lceil d/2\rceil -1}_n$ is a combinatorial ball that shares no common facets with $-B^{d,\lceil d/2\rceil -1}_n$), define $\Delta^{d}_{n+1}$ as the complex obtained from $\Delta^{d}_n$ by replacing $B^{d,\lceil d/2\rceil -1}_n$ with $\partial B^{d,\lceil d/2\rceil -1}_n * v_{n+1}$ and $-B^{d,\lceil d/2\rceil -1}_n$ with $\partial (-B^{d,\lceil d/2\rceil -1}_n) * (-v_{n+1})$.
	\end{itemize}
\end{definition}

To get the feel for this construction, we start by computing several explicit examples of complexes produced by Definition \ref{main-constr}. Note that the join of any simplicial complex
with the void complex $\emptyset$ is the void complex. Hence by induction, for all $d\geq 1$, $B^{d, 0}_n$ is the simplex on the vertex set $\{-v_1, v_{n-d+1}, v_{n-d+2}, \dots, v_n\}$. In particular, the link of $v_{i+1}$ in $\Delta^2_{i+1}$ is $\partial B^{2,0}_i$, the boundary of the triangle $\{-v_1,v_{i-1},v_i\}$; similarly, the link of $-v_{i+1}$ is the boundary of $\{v_1,-v_{i-1},-v_i\}$. It follows that $\Delta^2_n$ is obtained from $\partial\C^*_3$ by symmetric stacking --- an operation defined in \cite{KNNZ}. 

To get a handle on $B^{3,1}_n$, it is worth noting that for $i\leq \lceil d/2\rceil-1$, Definition \ref{main-constr} implies that
\begin{equation*}
	\begin{split}
	B^{d, i}_n &= \left( B^{d-2, i}_{n-2}*(v_{n-1}, v_n)  \cup  (-B^{d-2, i-1}_{n-2})* (-v_{n-1}, v_n)\right) \\
	&\quad \cup\left( (-B^{d-2,i-1}_{n-2})*(-v_{n-1}, -v_{n}) \cup  B^{d-2, i-2}_{n-2}*(v_{n-1}, -v_n)\right)\\
	&=\left( B^{d-2, i}_{n-2}*(v_{n-1}, v_n) \right)  \cup  \left( (-B^{d-2, i-1}_{n-2})* (v_n, -v_{n-1}, -v_n)\right)\cup\left( B^{d-2, i-2}_{n-2}*(v_{n-1}, -v_n)\right).
	\end{split}
\end{equation*}
If we let $d=3$ and $i=1$, then
$$B^{3,1}_n=\Big( (v_{n-2}, v_{n-3}, \dots, v_1, -v_{n-2}, -v_{n-3}, \dots, -v_1)* (v_{n-1}, v_n) \Big) \cup \Big( (v_1, -v_{n-2})*(v_n, -v_{n-1}, -v_n)\Big) .$$
In particular, $B^{3,1}_n$ is a combinatorial $3$-ball that has $2n-3$ facets; it is cs-$1$-neighborly w.r.t.~$V_n$ and $1$-stacked. This complex plays a key role in Jockusch's construction \cite{Jockusch95} of cs-$2$-neighborly $3$-spheres: in fact, Jockusch's $3$-spheres are exactly the complexes $\Delta^3_n$, for $n\geq 4$.

Several other easy consequences of Definition \ref{main-constr} are: if $i\leq \lfloor d/2 \rfloor$ and if $B^{d,i}_n$ is well-defined, then every facet of $B^{d,i}_n$ contains either $v_n$ or $-v_n$, and no face of $B^{d,i}_n$ contains two antipodal vertices. This latter property and the fact that  $\Delta^{d}_{d+1}=\partial \C_{d+1}^*$  is cs-$(d+1)$-neighborly implies that $\Delta^{d}_{d+1}$ does contain all balls $B^{d,i}_{d+1}$ for $i\leq \lceil d/2\rceil$, so we can at least start executing the algorithm in Definition \ref{main-constr} for all $d$ (once we have taken care of smaller dimensions).

Our remaining task is to show that this algorithm never gets stuck and that its output, $\Delta^{d}_n$, is a cs combinatorial $d$-sphere with $2n$ vertices that is cs-$\lceil d/2\rceil$-neighborly. To start, we verify in Lemma~\ref{lm: property of B^{2k-1,k-1}_n} that if Definition~\ref{main-constr} allowed us to reach a point where the complex $\Delta^{d}_n$ and its subcomplex $B^{d,\lceil d/2\rceil-1}_n$ were constructed, then all complexes produced by the definition up to that point satisfy all the assumptions of Lemma~\ref{lm: induction method}. This allows us to advance one more step and construct  $\Delta^{d}_{n+1}$. We then need to show that $\Delta^{d}_{n+1}$ produced in this way contains $B^{d,\lceil d/2\rceil-1}_{n+1}$. 
\begin{lemma}\label{lm: property of B^{2k-1,k-1}_n}
	If the algorithm reached the $(d,n)$-th step and produced a pair $\Delta^{d}_n \supseteq B^{d,\lceil d/2 \rceil-1}_n$, then $\Delta^{d}_n$ is a cs combinatorial $d$-sphere with vertex set $V_n$ that is cs-$\lceil d/2 \rceil$-neighborly while the complexes $B^{d,i}_n$ (for $0\leq i\leq \lceil d/2 \rceil$) are combinatorial $d$-balls that are cs-$i$-neighborly w.r.t.~$V_n$ and $i$-stacked. Furthermore, they satisfy the following ``nesting property'': $-B^{d,i-1}_n\subseteq B^{d,i}_n$ (for all $i\leq \lceil d/2 \rceil$). Finally, for all $i\leq \lfloor d/2 \rfloor$, $B^{d,i}_n$ and $-B^{d,i}_n$ share no common facets.
\end{lemma}
\begin{proof}
	We verify the properties of $\Delta^{d}_n$ and $B^{d,i}_n$ by induction on the dimension and the number of vertices. Since $\Delta^{d}_{d+1}=\partial\C^*_{d+1}$, the complex $\Delta^{d}_{d+1}$ is indeed a cs combinatorial $d$-sphere with $V(\Delta^{d}_{d+1})=V_{d+1}$ that is cs-$\lceil d/2 \rceil$-neighborly.
	Furthermore, for any $m\geq 2$,  $B^{1,0}_m$ and $B^{1,1}_m$ satisfy all the conditions: $B^{1,0}_m$ is an edge, and hence it is a cs-$0$-neighborly and $0$-stacked $1$-ball, while $B^{1,1}_m$ is a path on $V_m$, and hence a cs-$1$-neighborly (w.r.t.~$V_m$) and $1$-stacked $1$-ball. Finally, $-B^{1,0}_m=\overline{v_1(-v_m)}$ and $$B^{1,1}_m=\Delta^{1}_m\backslash B^{1,0}_m=(-v_1, -v_2, \dots, -v_m, v_1, v_2, \dots, v_m)\supseteq -B^{1,0}_m.$$
	
	For the inductive step, since the algorithm reached the $(d,n)$-th step, we can
	assume that the complexes $B^{d-1,j}_{n-1}$ (for $j\leq \lceil (d-1)/2 \rceil$) satisfy all the conditions of the lemma and that if $n>d+2$, then $\Delta^{d}_{n-1}$ is a cs combinatorial $d$-sphere (with vertex set $V_{n-1}$) that is cs-$\lceil d/2 \rceil$-neighborly. We now show that then the same holds for $B^{d,i}_n$ (for all $0\leq i\leq \lceil d/2 \rceil$) and $\Delta^{d}_{n}$. 
	
	We start with the nesting property. By definition, for all $i\leq \lfloor d/2\rfloor=\lceil (d-1)/2 \rceil$,
\begin{align*}
	-B^{d,i-1}_{n}&=\left( (-B^{d-1, i-1}_{n-1})*(-v_n)\right) \cup \left( B^{d-1, i-2}_{n-1}*v_n\right),\qquad \mbox{and}\\
	B^{d,i}_{n}&=\left( (-B^{d-1, i-1}_{n-1})*(-v_n)\right) \cup \left( B^{d-1, i}_{n-1}*v_n\right).
\end{align*}
	By the inductive hypothesis, $B^{d-1,i-2}_{n-1}$ is a subcomplex of $B^{d-1, i}_{n-1}$. Hence $-B^{d,i-1}_n\subseteq B^{d, i}_n$ for $i\leq \lfloor d/2 \rfloor$. We will treat the case of $d=2k-1, \ i=k$ a bit later. 
	
	Next, it follows from the nesting property and the definition of $B^{d-1,j}_n$ that $B^{d-1, i}_{n-1}$ and $B^{d-1, i-1}_{n-1}$ share no common facets for all $1\leq i\leq \lfloor d/2 \rfloor \leq \lceil (d-1)/2 \rceil$, and hence neither do $B^{d,i}_n$ and $-B^{d,i}_n$. This result also implies the nesting property for $d=2k-1$ and $i=k$: since $B^{2k-1, k-1}_n$ and $-B^{2k-1, k -1}_n$ have no common facets,  $-B^{2k-1,k-1}_n\subseteq \Delta^{2k-1}_n\backslash B^{2k-1, k-1}_n=B^{2k-1,k}_n$.
	
	We now show that $B^{d,i}_n$ is an $i$-stacked combinatorial $d$-ball. For the case of $i\leq \lfloor d/2\rfloor$, recall that $B^{d,i}_n$ is the union of $D_1=B^{d-1, i}_{n-1}*v_{n}$ and $D_2=(-B^{d-1, {i-1}}_{n-1})*(-v_{n})$. By the inductive hypothesis and Lemma \ref{lm: interior faces of join}, $D_1$ and $D_2$ are combinatorial $d$-balls. The nesting propery implies that $D_1\cap D_2=-B^{d-1,i-1}_{n-1}$ and by the inductive hypothesis it is also a combinatorial $(d-1)$-ball. Since it is contained in $\partial D_1 \cap \partial D_2$, Lemma \ref{lm: interior faces of join} guarantees  that $B^{d,i}_n=D_1\cup D_2$ is a combinatorial $d$-ball. Furthermore, by the inductive hypothesis and Lemma \ref{lm: stackedness via union and join}, $D_1$ is $i$-stacked, $D_2$ is $(i-1)$-stacked and the intersection is $(i-1)$-stacked. Thus the union $B^{d,i}_n$ is $i$-stacked. 
	
    Next we turn to treating cs-neighborliness: we show that for all $i\leq \lfloor d/2\rfloor$, $B^{d,i}_n$ is cs-$i$-neighborly w.r.t.~$V_n$ while $\Delta^{d}_n$ is a cs combinatorial $d$-sphere that is cs-$\lceil d/2\rceil$-neighborly w.r.t~$V_n$. The statement about $B^{d,i}_n$ follows easily from the definition of $B^{d,i}_n$ and the inductive hypothesis asserting that $B^{d-1,j}_{n-1}$ is cs-$j$-neighborly w.r.t.~$V_{n-1}$ for $j=i-1,i$. Now, if $n=d+1$, then $\Delta^{d}_n=\partial\C^*_{d+1}$, so it is a cs combinatorial $d$-sphere on $V_n$ that is cs-$\lceil d/2\rceil$-neighborly, and if $n>d+1$, then the inductive hypothesis on $\Delta^{d}_{n-1}$ along with Lemma~\ref{lm: induction method} and the established properties of $B^{d,\lceil d/2\rceil-1}_{n-1}$ imply that $\Delta^{d}_n$ is a cs combinatorial $d$-sphere on $V_n$ that is cs-$\lceil d/2\rceil$-neighborly. 
		
		Finally, we discuss the case of $d=2k-1, \ i=k$. Since $B^{2k-1,k}_n=\Delta^{2k-1}_n\backslash B^{2k-1, k-1}_n$, we obtain from the previous paragraph and Lemma \ref{lm: property of complement} that $B^{2k-1,k}_n$ is a combinatorial $(2k-1)$-ball that is cs-$k$-neighborly and $k$-stacked. This completes the proof.
\end{proof}

It now remains to show that $B^{d,\lceil d/2\rceil -1}_{n+1}$ is a subcomplex of $\Delta^{d}_{n+1}$.  (\emph{Recall that our assumptions include that $B^{d-1,\lceil (d-1)/2\rceil -1}_m$ is a subcomplex of $\Delta^{d-1}_m$ for all $m\geq d$ and that $B^{d, \lceil d/2\rceil -1}_{n}$ is a subcomplex of $\Delta^{d}_{n}$.}) To facilitate the proof of this result, we rely on the following lemma; its proof is an immediate consequence of the definition of $B^{d,i}_n$, Lemma \ref{lm: interior faces of join}, and the nesting property.

\begin{lemma}\label{lm: partial B}
	Under the assumptions of Lemma \ref{lm: property of B^{2k-1,k-1}_n}, for all $d\geq 2$ and $0\leq i\leq \lfloor d/2 \rfloor$, 
	$$\partial B^{d, i}_n=\left( \partial B^{d-1, i}_{n-1}*v_n\right)   \cup \left( \partial (-B^{d-1, i-1}_{n-1})*(-v_{n})\right) \cup \left( B^{d-1, i}_{n-1}\backslash (-B^{d-1, i-1}_{n-1})\right) .$$
\end{lemma}
A combinatorial sphere $\Delta$ is called \emph{$k$-stacked} if it is the boundary complex of a combinatorial $k$-stacked ball.
\begin{corollary}\label{cor: Delta is k-stacked}
	Under the assumptions of Lemma \ref{lm: property of B^{2k-1,k-1}_n}, for all $k\geq 1$ and $n\geq 2k+1$, $\partial B^{2k, k}_n=\Delta^{2k-1}_n$. In particular, $\Delta^{2k-1}_n$ is $k$-stacked.
\end{corollary}
\begin{proof}
	By definition, the complex $\Delta^{2k-1}_n$ is obtained from $\Delta^{2k-1}_{n-1}=B^{2k-1, k-1}_{n-1}\cup B^{2k-1,k}_{n-1}$ by replacing $\pm B^{2k-1, k-1}_{n-1}$ with $\pm (\partial B^{2k-1,k-1}_{n-1}* v_n)$. On the other hand, Lemma \ref{lm: partial B} and the fact that $\partial B^{2k-1, k}_n=\partial B^{2k-1, k-1}_n$ yield that $$\partial B^{2k, k}_n=\left( \partial B^{2k-1, k-1}_{n-1}*v_n\right)   \cup \left( \partial (-B^{2k-1, k-1}_{n-1})*(-v_{n})\right) \cup \left( B^{2k-1, k}_{n-1}\backslash (-B^{2k-1, k-1}_{n-1})\right) .$$
	Hence $\partial B^{2k, k}_n =\Delta^{2k-1}_n$. Since by Lemma \ref{lm: property of B^{2k-1,k-1}_n}, the ball $B^{2k, k}_n$ is $k$-stacked , the sphere $\Delta^{2k-1}_n$ is also $k$-stacked.
\end{proof}
	
We are now in a position to show that the ball $B^{d, \lceil d/2\rceil -1}_{n}$ is indeed a subcomplex of $\Delta^d_n$. This will require one additional lemma.

\begin{lemma}\label{lm: inclusion}
	Under the assumptions of Lemma \ref{lm: property of B^{2k-1,k-1}_n}, for all $d\geq 2$, $n\geq d+1$, and $0\leq i\leq j \leq \lfloor d/2 \rfloor$, the following inclusion holds: $B^{d-1, i}_n\subseteq \partial B^{d, j}_n$.
\end{lemma}
\begin{proof}
	The proof is by induction on $d$. We start by checking the statement in the following base cases.
	\begin{enumerate}
		\item If $j=0$ (and hence $i=0$), then for any $d\geq 2$, $$\partial B^{d, 0}_n=\partial \overline{\{-v_1,  v_{n-d+1}, \dots, v_n\}}\supseteq \overline{\{-v_1, v_{n-d+2}, \dots, v_n\}}=B^{d-1, 0}_n.$$
		
		\item If $i=0$, then by Lemmas \ref{lm: partial B} and \ref{cor: Delta is k-stacked}, for any $j>0$ and $d\geq 2j$,
		$$\partial B^{d, j}_n \supseteq \partial B^{d-1, j}_{n-1}*v_n \supseteq \dots \supseteq \partial B^{2j, j}_{n-d+2j}* \overline{\{v_{n-d+2j+1}, \dots, v_n\}}= \Delta^{2j-1}_{n-d+2j}*\overline{\{v_{n-d+2j+1}, \dots, v_n\}}.$$
		Since $B^{d-1, 0}_n=B^{2j-1,0}_{n-d+2j}*\overline{\{v_{n-d+2j+1}, \dots, v_n\}}$ and $\Delta^{2j-1}_{n-d+2j}\supseteq B^{2j-1, 0}_{n-d+2j}$, we conclude that $B^{d-1, 0}_n \subseteq \partial B^{d,j}_n$.
		\item If $j\geq 1$ and $d=2j$, then Lemma \ref{cor: Delta is k-stacked} implies that for any $i\leq j$,  $\partial B^{2j,j}_n=\Delta^{2j-1}_n\supseteq B^{2j-1, i}_n$.
	\end{enumerate}

For the inductive step we can thus assume that $1\leq j<d/2$ (equivalently, that $1\leq j\leq \lfloor (d-1)/2\rfloor$), and that the statement holds for $d-1$ and all $0\leq i\leq j\leq \lfloor (d-1)/2\rfloor$. Then
\begin{equation*}
	\begin{split}
	B^{d-1, i}_n &= \left( B^{d-2, i}_{n-1} * v_n\right)  \cup \left( (-B^{d-2, i-1}_{n-1})*(-v_n)\right) \\
	&\stackrel{(*)}{\subseteq} \left( \partial B^{d-1, j}_{n-1} *v_n\right)  \cup \left( (-\partial B^{d-1, j-1}_{n-1})*(-v_n)\right) \\
	&\stackrel{(**)}{\subseteq} \partial B^{d, j}_n.	
	\end{split}
\end{equation*}
Here $(*)$ follows from the inductive hypothesis and $(**)$ follows from Lemma \ref{lm: partial B}. This completes the proof.
\end{proof}

An immediate corollary of Lemma \ref{lm: inclusion} is that $B^{d,\lceil d/2\rceil-1}_{n+1} \subseteq \Delta^d_{n+1}$:
\begin{corollary}\label{lm: B is subcomplex}
	For all $d\geq 2$ and $n\geq d+1$, $$B^{d, \lceil d/2\rceil-1}_{n+1}\subseteq \left( \partial B^{d, \lceil d/2 \rceil -1}_n *v_{n+1}\right)  \cup \left( \partial (-B^{d, \lceil d/2 \rceil -1}_n) *(-v_{n+1})\right) \subseteq \Delta^d_{n+1}.$$
\end{corollary}
\begin{proof}
	The first inclusion is by definition of $B^{d, \lceil d/2 \rceil-1}_n$ and Lemma \ref{lm: inclusion}, and the second inclusion is by definition of $\Delta^d_{n+1}$.
\end{proof}

Corollary \ref{lm: B is subcomplex} together with Lemma \ref{lm: property of B^{2k-1,k-1}_n} completes the proof that the construction described in Definition \ref{main-constr} never gets stuck and that for every $d\geq 2$ and $n\geq d+1$ it outputs a cs combinatorial $d$-sphere with vertex set $V_n$ that is cs-$\lceil d/2 \rceil$-neighborly. We summarize this in the theorem below:
\begin{theorem}\label{thm}
	For all $d\geq 2$ and $n\geq d+1$, the complex $\Delta^d_n$ is a cs combinatorial $d$-sphere with vertex set $V_n$ that is cs-$\lceil d/2\rceil$-neighborly.
\end{theorem}

\begin{remark} \label{rm:neighborly-i-spheres-construction}
Let $1\leq i\leq \lceil d/2 \rceil$.
 Since by nesting property, 
$$B^{d,i-1}_n\subseteq B^{d, \lceil d/2 \rceil-1}_n \cup \left(-B^{d, \lceil d/2 \rceil-1}_n\right),$$ it follows from	Corollary \ref{lm: B is subcomplex} that $B^{d,i-1}_n\subseteq \Delta^d_n$. This allows us to construct cs-$i$-neighborly spheres that are not cs-$(i+1)$-neighborly for {\em all} $1\leq i\leq \lceil d/2 \rceil$. Indeed, 
Lemmas \ref{lm: induction method} and \ref{lm: property of B^{2k-1,k-1}_n} along with Theorem \ref{thm} imply that the complex  $\Delta^{d,i}_{n+1}$ obtained from $\Delta^d_n$ by replacing the subcomplexes $\pm B^{d, i-1}_n$ with $\pm \big(\partial B^{d, i-1}_n* v_{n+1}\big)$ is a cs combinatorial $d$-sphere that is cs-$i$-neighborly. To see that $\Delta^{d,i}_{n+1}$ is not cs-$(i+1)$-neighborly, note that the inductive proof of Lemma \ref{lm: property of B^{2k-1,k-1}_n} in fact shows that while $B^{d, i-1}_n$ is cs-$(i-1)$-neighborly, it is not cs-$i$-neighborly. In particular, $\partial B^{d, i-1}_n$ is not cs-$i$-neighborly, and hence $\Delta^d_{n+1}$ is not cs-$(i+1)$-neighborly.
\end{remark}

\section{Properties} 
It is our hope that the spheres $\Delta^d_n$ constructed in this paper will find many other applications. With this in mind, in this section we discuss several additional properties that these spheres possess. One important property is that certain face links in $\Delta^d_n$ are also highly cs-neighborly.

\begin{proposition}\label{lm: link of v_{n-1}v_n} For all $k\geq 2$ and $n\geq 2k-1$,
	$$\lk(v_nv_{n+1},\Delta^{2k-1}_{n+1})=\Delta^{2k-3}_{n-1}=\lk(v_{n}v_{n+1}v_{n+2}, \Delta^{2k}_{n+2}).$$
\end{proposition}
\begin{proof} The assertion holds if $n=2k-1$ since in this case all complexes appearing in the statement are the boundary complexes of cross-polytopes. Thus assume $n\geq 2k$.
	By definition of $\Delta^{2k-1}_{n+1}$, Lemma \ref{lm: partial B}, and Lemma \ref{cor: Delta is k-stacked}, $$\lk(v_nv_{n+1}, \Delta^{2k-1}_{n+1})=\lk(v_n, \partial B^{2k-1, k-1}_{n})=\partial B^{2k-2, k-1}_{n-1}=\Delta^{2k-3}_{n-1}.$$
	Similarly, $$\lk(v_{n}v_{n+1}v_{n+2}, \Delta^{2k}_{n+2})=\lk(v_{n}v_{n+1}, \partial B^{2k, k-1}_{n+1})=\lk(v_{n}, \partial B^{2k-1, k-1}_n)=\partial B^{2k-2, k-1}_n=\Delta^{2k-3}_{n-1}.$$
This completes the proof.\end{proof}

In view of Proposition \ref{lm: link of v_{n-1}v_n}, it is natural to ask whether the link of $v_{n+1}$ in $\Delta^{2k}_{n+1}$ is $\Delta^{2k-1}_n$, and, more generally, whether $\Delta^{2k}_{n+1}$ is a suspension. This is especially relevant to our discussion since suspensions of cs-$k$-neighborly spheres are also cs-$k$-neighborly; in particular, both complexes $\Sigma \Delta^{2k-1}_n$ and $\Delta^{2k}_{n+1}$ are cs combinatorial $2k$-spheres that are cs-$k$-neighborly w.r.t~$V_{n+1}$. Are they isomorphic? The following proposition shows that the answer is no.


\begin{proposition}\label{cor: no suspension}
	If $n\geq 2k+1$, then the complex $\Delta^{2k}_{n+1}$ is not a suspension. In particular, $\Delta^{2k}_{n+1}$ is not isomorphic to $\Sigma \Delta^{2k-1}_n$.
\end{proposition}
\begin{proof}
We will rely on the following three observations.
\begin{enumerate}
	\item If $\Delta=\Sigma\Gamma$ and $F\in \Gamma$, then $\lk(F, \Delta)=\Sigma\lk(F, \Gamma)$.
	\item For $k\geq 2$ and $n\geq 2k+1$, $\Delta^{2k-3}_{n-2}$ is not a suspension. 
	
	To see this, recall that $\Delta^{2k-3}_{n-2}$ is cs-$(k-1)$-neighborly. Now, a cs $(2k-4)$-sphere that is not the boundary of the cross-polytope can be at most cs-$(k-2)$-neighborly. Hence, if $\Delta^{2k-3}_{n-2}$ were a suspension, it would also be at most cs-$(k-2)$-neighborly. 
	
	\item For $\ell \geq 1$ and $m\geq 2\ell +1$, the complex $\partial B^{2\ell, \ell-1}_m$ is not a cs complex.
	
	If $\partial B^{2\ell, \ell-1}_m$ were a cs complex, then the link $\lk(v_m, \partial B^{2\ell, \ell-1}_m)=\partial B^{2\ell-1, \ell-1}_{m-1}$ would coincide with $-\lk(-v_m, \partial B^{2\ell, \ell-1}_m)=-\partial (-B^{2\ell-1, \ell-2}_{m-1})=\partial B^{2\ell-1, \ell-2}_{m-1}$. Since the balls $B^{2\ell-1, \ell-1}_{m-1}$ and $B^{2\ell-1, \ell-2}_{m-1}$ share no common facets, their union would then be a $(2\ell-1)$-sphere strictly contained in the $(2\ell-1)$-sphere $\Delta^{2\ell-1}_m$, which is impossible.
\end{enumerate}

We are now ready to prove the proposition. If $k=1$, then $\Delta^2_{n+1}$ is obtained from $\partial C^*_3$ by symmetric stacking, and so $\Delta^2_{n+1}$ is not a suspension for $n\geq 3$. Thus let $k\geq 2$, and assume by contradiction that $\Delta^{2k}_{n+1}$ is a suspension with suspension vertices $\pm v_i$ for some $i\leq n+1$. 

If $i<n-1$, then by Observation 1, the link of $v_{n-1}v_nv_{n+1}$ in $\Delta^{2k}_{n+1}$ must be a suspension. However, by Proposition \ref{lm: link of v_{n-1}v_n}, $\lk(v_{n-1}v_nv_{n+1}, \Delta^{2k}_{n+1})=\Delta^{2k-3}_{n-2}$, which according to Observation 2 is not a suspension.

If $i=n+1$, then $\lk(v_{n+1}, \Delta^{2k}_{n+1})$ must equal $\lk(-v_{n+1}, \Delta^{2k}_{n+1})$. In particular, $\lk(v_{n+1}, \Delta^{2k}_{n+1})=\partial B^{2k, k-1}_n$ must be a cs complex, which contradicts Observation 3. 

If $i=n$, then by Observation 1, $\lk(v_{n+1}, \Delta^{2k}_{n+1})=\partial B^{2k, k-1}_n$ must be a suspension with suspension vertices $\pm v_n$. However, since $-B^{2k-1, k-2}_{n-1}$ is strictly contained in $B^{2k-1, k-1}_{n-1}$, the complexes $\lk(v_n, \partial B^{2k,k-1}_n)=\partial B^{2k-1, k-1}_{n-1}$ and $\lk(-v_n, \partial B^{2k, k-1}_n)=\partial (-B^{2k-1, k-2}_{n-1})$ are not equal.
	
	Finally, if $i=n-1$, then by Observation 1, we must have
	$$\lk(v_{n-1}, \lk(v_nv_{n+1}, \Delta^{2k}_{n+1}))=\lk(-v_{n-1}, \lk(v_nv_{n+1}, \Delta^{2k}_{n+1})).$$
	The proof of Proposition \ref{lm: link of v_{n-1}v_n} implies that the former complex is $\Delta^{2k-3}_{n-2}$ while the latter complex is $\partial (-B^{2k-2, k-2}_{n-2})$. These complexes cannot be equal since one is cs and the other one is non-cs.
\end{proof}

\begin{remark} \label{rm:special-ball}
The ball $B^{2k,k}_{n+1}$ is ``special" in the following sense: while $B^{d,i}_{n+1}\subseteq \Delta^d_{n+1}$ for all $d\geq 1$, $i\leq \lceil d/2\rceil-1$, and $n\geq d$ (see Remark \ref{rm:neighborly-i-spheres-construction}), the ball $B^{2k,k}_{n+1}$ 
is {\bf not} a subcomplex of $\Delta^{2k}_{n+1}$ for $n>2k$. Indeed, a straightforward computation using that $\Delta^{2k-1}_n=B^{2k-1,k}_n\cup B^{2k-1,k-1}_n$ along with the definition of $B^{2k,k}_{n+1}$ implies that $B^{2k,k}_{n+1}\cup (-B^{2k,k}_{n+1})=\Sigma \Delta^{2k-1}_{n}$. Our claim that $B^{2k,k}_{n+1}\not\subseteq \Delta^{2k}_{n+1}$ (for $n\geq 2k+1$) then follows from Proposition \ref{cor: no suspension}. However, in the case of $n=2k$, the same computation shows that $B^{2k, k}_{2k+1} \cup (-B^{2k, k}_{2k+1})=\Sigma \Delta^{2k-1}_{2k}=\Sigma(\partial\C^*_{2k})=\partial\C^*_{2k+1}=\Delta^{2k}_{2k+1}$.
\end{remark}

To close this section, we show that while $\Delta^{2k}_{n+1}$ is not the suspension of $\Delta^{2k-1}_n$, the complexes $\Delta^{2k}_{n}$ and $\Delta^{2k-1}_n$ are closely related:
\begin{proposition} \label{prop: inclusion}
	Let $k\geq 2$ and $n\geq 2k+1$. Then $\Delta^{2k}_n\supseteq \Delta^{2k-1}_n$.
\end{proposition}
\begin{proof}
To prove the statement, it suffices to construct a combinatorial $2k$-ball $D_n$ that satisfies the following properties:
\begin{itemize}
\item  $D_n \supseteq -B^{2k,k-1}_n$.
\item $D_n$ and $-D_n$ share no common facets and their union is $\Delta^{2k}_n$; in particular, $\pm D_n\subseteq \Delta^{2k}_n$.
\item $\partial D_n=\partial(-D_{n})=\Delta^{2k-1}_{n}$, and so $\Delta^{2k-1}_n \subseteq \Delta^{2k}_n$.
\end{itemize}

We construct $D_n$ by induction on $n$. For the base case, define $D_{2k+1}:=B^{2k,k}_{2k+1}$. 
This ball has all the desired properties: this follows from 
Lemma \ref{lm: property of B^{2k-1,k-1}_n}, Corollary \ref{cor: Delta is k-stacked}, and Remark \ref{rm:special-ball}.

For the inductive step, assume that there exists a combinatorial $2k$-ball $D_n$ that satisfies all of the above properties. In particular, $-D_n$ contains $B^{2k,k-1}_n$ as a subcomplex.
Recall also that  $\partial B^{2k, k-1}_n \supseteq B^{2k-1, k-1}_n$ (see Lemma \ref{lm: inclusion}). 
Thus, the following complex is well-defined:
\begin{equation}\label{-D_n}
		D_{n+1}:=\left( (-D_n) \backslash B^{2k, k-1}_n \right) \cup \left( (-B^{2k-1, k-1}_n)*(-v_{n+1})\right) \cup \left( \big((\partial B^{2k, k-1}_n)\backslash B^{2k-1, k-1}_n\big) *v_{n+1}\right).
	\end{equation}

This definition along with the inductive assumption that $D_n$ and $-D_n$ share no common facets implies that $D_{n+1}$ and $-D_{n+1}$ share no common facets either. Further, since $\partial B^{2k, k-1}_n$ contains $B^{2k-1,k-2}_n$ (see Lemma \ref{lm: inclusion}) and since $B^{2k-1,k-2}_n$ and $B^{2k-1, k-1}_n$ share no common facets (this follows from the nesting property), equation \eqref{-D_n} guarantees that 
\[
D_{n+1} \supseteq \left((-B^{2k-1, k-1}_n)*(-v_{n+1})\right) \cup \left(B^{2k-1,k-2}_n*v_{n+1}\right)=-B^{2k,k-1}_{n+1}.
\]

Recall that $\Delta^{2k}_{n+1}$ is obtained from $\Delta^{2k}_{n}$ by replacing $\pm B^{2k, k-1}_n$ with $\pm (\partial B^{2k, k-1}_n *v_{n+1})$. This together with the definition of $D_{n+1}$ and the inductive hypothesis asserting that $D_n\cup (-D_n)=\Delta^{2k}_n$ implies that $D_{n+1}\cup (-D_{n+1})=\Delta^{2k}_{n+1}$.

It only remains to show that $D_{n+1}$ is a combinatorial $2k$-ball with $\partial D_{n+1}=\partial(-D_{n+1})=\Delta^{2k-1}_{n+1}$. We use eq.~\eqref{-D_n} and the inductive hypothesis that $\partial(-D_n)=\Delta^{2k-1}_n$. First, the boundary of $\partial B^{2k, k-1}_n * v_{n+1}$ coincides with the boundary of $B^{2k, k-1}_n$. Thus, replacing the subcomplex $B^{2k, k-1}_n$ of the combinatorial ball $-D_n$ with $\partial B^{2k, k-1}_n * v_{n+1}$ results in a combinatorial ball $D'_{n+1}$ that has the same boundary as $-D_n$. The ball $D'_{n+1}$ has $v_{n+1}$ as an interior vertex whose link in $D'_{n+1}$ is $\partial B^{2k, k-1}_n$. Now, $B^{2k-1,k-1}_n$ is contained in both $\partial B^{2k, k-1}_n$ and $\partial D'_{n+1}=\partial(-D_n)=\Delta^{2k-1}_n$. Since $v_{n+1}$ is an interior vertex of $D'_{n+1}$, removing $B^{2k-1,k-1}_n*v_{n+1}$ from $D'_{n+1}$ results in a combinatorial ball $D''_{n+1}$ whose boundary is obtained from that of $-D_n$ by replacing $B^{2k-1,k-1}_n$ with $\partial B^{2k-1,k-1}_n  * v_{n+1}$. Finally, the balls $(-B^{2k-1, k-1}_n)*(-v_{n+1})$ and $D''_{n+1}$ intersect along $-B^{2k-1, k-1}_n$, which is a subcomplex of their boundaries. This yields that $D_{n+1}=D''_{n+1} \cup \left((-B^{2k-1, k-1}_n)*(-v_{n+1})\right)$ is a combinatorial ball. Further, the boundary of $D_{n+1}$ is obtained from the boundary of $D''_{n+1}$ by replacing $-B^{2k-1, k-1}_n$ with $\partial (-B^{2k-1, k-1}_n)*(-v_{n+1})$. We conclude that $\partial D_{n+1}=\Delta^{2k-1}_{n+1}$.
\end{proof}

Observe that the $2k$-ball $D_n$ constructed in the proof of Proposition \ref{prop: inclusion} is $k$-stacked. (Indeed, since the boundary of $D_n$ is a cs-$k$-neighborly complex w.r.t.~$V_n$, no face of $D_n$ of dimension $\leq k-1$ is an interior face of $D_n$, and so all interior faces of $D_n$ have dimension $\geq k=2k-k$.) Thus one curious consequence of Proposition \ref{prop: inclusion} and Corollary \ref{cor: Delta is k-stacked} is that for $n>2k+1$, $\Delta^{2k-1}_n$ is the boundary complex of at least four distinct $k$-stacked $2k$-balls, namely, $D_n$, $-D_n$, $B^{2k,k}_n$, and $-B^{2k,k}_n$.

\section{Closing remarks and questions}
We close with a few open questions.

The spheres $\Delta^{d}_n$ we constructed here are cs combinatorial spheres that are cs-$\lceil d/2 \rceil$-neighborly. Hence, according to \cite{McMShep}, for $d\geq 3$ and $n\geq d+3$, they are not the boundary complexes of cs polytopes. In fact, it follows from results of Pfeifle \cite[Chapter 10]{Pfeifle} that they are not even \emph{cs fans} (at least for $d$ big enough). However, these spheres might still possess some additional ``liked by all" properties:
\begin{problem} Let $d\geq 3$ and $n\geq d+3$.
	Are the spheres $\Delta^{d}_n$ vertex decomposable or at least shellable? Are they realizable as (non-cs) fans? Are they even realizable as boundary complexes of (non-cs) polytopes?
\end{problem}

It is well-known that there are many (non-cs) neighborly polytopes and neighborly spheres. For example, the number of combinatorial types of $\lceil d/2\rceil$-neighborly $(d+1)$-polytopes with $n$ vertices is at least $n^{\frac{d-1}{2}n(1+o(1))}$ for $d> 1$ and $n\to \infty$, see \cite[Section 6]{Padrol-13}. Are there many cs $d$-spheres that are cs-$\lceil d/2\rceil$-neighborly?
\begin{problem}
Find many new constructions of cs combinatorial $d$-spheres that are cs-$\lceil d/2\rceil$-neighborly.
\end{problem}

Finally, it is worth mentioning that in a contrast with cs combinatorial spheres, there exist cs combinatorial $2k$-manifolds with $n>2(2k+1)$ vertices that are cs-$(k+1)$-neighborly. (The interest in such complexes arises in part from Sparla's conjecture \cite{Sparla97,Sparla98} that posits an upper bound on the Euler characteristic of cs combinatorial $2k$-manifolds with $2n$ vertices; see \cite{N05} for some results on this conjecture.) One construction of such an infinite family is given in \cite{KleeN}: for each $k\geq 1$, it produces a cs triangulation of the product of two $k$-dimensional spheres with $4k+4$ vertices that is cs-$(k+1)$-neighborly. For additional constructions in low dimensions, see \cite{Effenberger,  Lutz, Sparla97}.

\begin{problem}
	Find new constructions of (infinite families of) cs combinatorial $2k$-manifolds that are cs-$(k+1)$-neighborly.
\end{problem}

\subsection*{Acknowledgments} 
We are grateful to the referee for suggesting an approach that greatly simplified and streamlined our main definitions and presentation.

{\small
	\bibliography{refs}
	\bibliographystyle{plain}
}

\end{document}